\documentclass[draft]{amsart}

\usepackage{amsmath, amssymb, epsfig, graphicx, url}
\newtheorem{thm}{Theorem}[section]
\newtheorem{lem}[thm]{Lemma}

\newtheorem{exam}[thm]{Example}
\newtheorem{defn}[thm]{Definition}
\newenvironment{ex}{\begin{exam}\rm}{ \end{exam}}
\newenvironment{de}{\begin{defn}\rm}{ \end{defn}}

\renewcommand{\leq}{\leqslant}
\renewcommand{\geq}{\geqslant}

\newcommand{\im}{\operatorname{im}}
\newcommand{\aut}{\mbox{\rm Aut\,}}

\newcommand{\orb}[2]{#1 #2}

\newcommand{\rank}{\operatorname{rank}}
\newcommand{\set}[2]{\ensuremath{\{\: #1 \: |\: #2 \:\}}}
\newcommand{\genset}[1]{\ensuremath{\langle\: #1 \:\rangle}}
\newcommand{\trans}{\mathcal{T}_{n}}
\newcommand{\sym}{\mathcal{S}_{n}}
\newcommand{\alt}{\mathcal{A}_{n}}
\newcommand{\norm}{N}
\newcommand{\agl}{\mbox{\rm AGL}}
\newcommand{\psl}{\mbox{\rm PSL}}
\newcommand{\pgl}{\mbox{\rm PGL}}

\newcommand{\pgaml}{\mbox{\rm P}\Gamma {\rm L}}

\newcommand{\psigl}{\mbox{\rm P}\Sigma {\rm L}}
\newcommand{\dihed}[1]{\ensuremath{D_{#1}}}

\newcommand{\M}{\mbox{\rm M}}

\begin{document}

%%%%%%%%%
%%%%%%%%%

%\begin{frontmatter}

\title[Groups, regular semigroups, and idempotent generated semigroups]{Groups that together with any transformation generate regular semigroups
or idempotent generated semigroups}
\author[Ara\'ujo, Mitchell, and Schnieder]{}

\maketitle

%%%%%%%%%
\begin{center}
{\small {\bf J. Ara\'ujo}\\
Universidade Aberta\\
R. da Escola Polit\'ecnica, 147\\
1269-001 Lisboa
Portugal\\\ \\
 Centro de \'Algebra, Universidade de Lisboa\\
Av.\ Prof.\ Gama Pinto, 2\\ 
1649-003 Lisboa, Portugal\\
Email: mjoao@ptmat.lmc.fc.ul.pt\\\ \\
{\bf J. D. Mitchell}\\
Mathematical Institute\\
 North Haugh, St Andrews, Fife KY16 9SS, Scotland\\
Email: jdm3@st-andrews.ac.uk\\\ \\
{\bf Csaba Schneider}\\
Centro de \'Algebra, Universidade de Lisboa\\
Av.\ Prof.\ Gama Pinto, 2, 1649-003 Lisboa, Portugal\\
Email: csaba.schneider@gmail.com}
\end{center}

\begin{abstract}
Let $a$ be a non-invertible transformation 
of a finite set and let $G$ be a group of permutations on that same set. Then  $\genset{G, a}\setminus G$ is a subsemigroup, consisting of all non-invertible 
transformations, in the semigroup generated by $G$ and $a$. Likewise, the conjugates $a^g=g^{-1}ag$ of $a$ by elements $g\in G$ generate a semigroup denoted $\genset{a^g\ |\ g\in G}$. 
We classify the finite permutation groups $G$ on a finite set
$X$ such that  the semigroups
$\genset{G,a}$,
$\genset{G, a}\setminus G$, and $\genset{a^g\ |\ g\in G}$ are regular for all
transformations of $X$.
We also classify the permutation groups 
$G$ on a finite set $X$ such that the semigroups $\genset{G, a}\setminus G$ and $\genset{a^g\ |\ g\in G}$ are  generated by their idempotents
for all non-invertible transformations of $X$.
\end{abstract}

\medskip

\noindent{\em Date:} 30 October 2009
{\em Corresponding author:} Csaba Schneider

%%%%%%%%%

%\begin{keyword}
%idempotent-generated semigroup  finite permutation group.
%\end{keyword}

%\end{frontmatter}

%%%%%%%%%
%%%%%%%%%

\newpage

\section{Introduction}

One of the fundamental aspects of 
the study of finite semigroups is the interplay between groups and idempotents.
There are many examples among the most famous structural theorems in semigroup 
theory, such as, the Rees Theorem~\cite{rees1940}, or McAlister's $P$-Theorem~\cite[Part II, Theorem~2.6]{mcal1974} (see also~\cite{munn1976}),  that show the large extent that the structure of a semigroup is shaped by a 
group acting in some way on an idempotent structure.

Another topic that has recently attracted a great deal of attention is a general problem that might be described as follows:
{\em classify all groups $G$ of permutations (possibly an adjective, 
such as, primitive, imprimitive, linear, rank $k$, added) of a set $X$
such that $G$ and an arbitrary singular mapping on $X$ generate (or give rise, in some other sense to) a semigroup 
with a given property}.  A mapping is \emph{singular} if it is not a bijection, and hence, non-invertible.
We offer the conjecture that the years to come will confirm
this as a mainstream topic in semigroup theory.
Such problems are considered, for instance, 
in \cite{ar}, \cite{steinberg}, \cite{cameron}, \cite{neumann}. In this paper, 
we offer a contribution to
this area by proving Theorems~\ref{main} and~\ref{th2} below.
Before stating these theorems we review some concepts and
introduce some notation. 

A \emph{permutation group of degree $n$} is just a subgroup of $\sym$.
Recall that an element $a$ of a semigroup $S$ is 
said to be an {\em idempotent}
if $aa=a$. We say that $S$ is {\em idempotent generated}
if it is generated by the set of its idempotent elements. 
An element $a$ is  {\em regular} if
there is some $b\in S$ such that $a=aba$. 
If $a$ is an idempotent, then $a=aaa$, which
shows that idempotent elements are always regular. A semigroup $S$ is
said to be {\em regular} if all its elements are regular.
The symbols $\trans$, $\alt$, and $\sym$ 
denote the semigroup of all transformations,
the group of all invertible even transformations,
and the group of all invertible transformations of the finite set 
$\{1,\ldots,n\}$, respectively. 
The permutation groups that appear in the following theorems
are described immediately before
Theorem~\ref{superlemma}. If $a$ is a transformation, and
$g$ is an invertible transformation, then $a^g$ denotes the conjugate
$g^{-1}ag$ of $a$ by $g$. If $Y$ is a subset of $\trans$, then $\genset Y$ 
denotes the semigroup generated by $Y$.

The main theorems of our paper are the following. 

\begin{thm}\label{main}
If $n\geq 1$ and $G$ is a subgroup of $\sym$, then the following are equivalent:
\begin{enumerate}
\item[(i)] 
The semigroup $\genset{G,a}\setminus G$ is idempotent generated for all $a\in \trans\setminus\sym$.
\item[(ii)] The semigroup 
$\genset{a^g \mid g\in G}$ is idempotent generated 
for all $a\in \trans\setminus\sym$.
\item[(iii)] One of the following is valid for $G$ and $n$:
\begin{enumerate}
\item[(a)] $n=5$ and $G\cong \agl(1,5)$; 
\item[(b)] $n=6$ and $G\cong \psl(2,5)$ or $\pgl(2,5)$;
\item[(c)] $G=\alt$ or $\sym$.
\end{enumerate}
\end{enumerate}
\end{thm}

\begin{thm}\label{th2}
If $n\geq 1$ and $G$ is a subgroup of $\sym$,  then the following are equivalent:
\begin{enumerate}
    \item[\rm (i)] The semigroup $\genset{G,a}$ is regular for all $a\in \trans\setminus\sym$.
    \item[\rm (ii)] The semigroup $\genset{G,a}\setminus G$ is regular for all $a\in \trans\setminus\sym$.
    \item[\rm (iii)] The semigroup $\genset{a^g \mid g\in G}$ is regular for all $a\in \trans\setminus\sym$.
\item[(iv)] One of the following is valid for $G$ and $n$:
    \begin{enumerate}
    \item[\rm (a)] $n=5$ and $G\cong C_5,\ \dihed{5},$ or $\agl(1,5)$;
    \item[\rm (b)] $n=6$ and $G\cong \psl(2,5)$ or $\pgl(2,5)$;
    \item[\rm (c)]  $n=7$ and $G\cong\agl(1,7)$;
    \item[\rm (d)] $n=8$ and $G\cong\pgl(2,7)$;
    \item[\rm (e)] $n=9$ and $G\cong\psl(2,8)$ or $\pgaml(2,8)$;
    \item[\rm (f)] $G=\alt$ or $\sym$.
\end{enumerate}
\end{enumerate}
\end{thm}
The proofs of Theorems \ref{main} and \ref{th2} will be given in Sections~\ref{proofsec} and \ref{classutg}.

%JDM shouldn't the following go at the beginning of the introduction or at least before the main theorems?

The line of research that this paper extends was initiated by  Howie in \cite{howie}.
Howie showed that  
the semigroup 
$\genset{\sym,a}\setminus \sym$ is idempotent generated and regular  
whenever $a$ 
is a singular transformation of $\trans$ with image of size $n-1$.
Later Symons~\cite{symons} and 
Levi and McFadden~\cite{lm} proved the following generalization.
For a singular transformation $a\in\trans\setminus \sym$,  the semigroups 
$\genset{a^g\ |\ g\in\sym}$ and $\genset{\sym,a}\setminus\sym$ coincide; in addition, this semigroup is idempotent generated and regular 
(see Lemma~\ref{antigo}).
Later Levi \cite{levi96} showed that this last theorem remains true
if we replace $\sym$ by the alternating group $\alt$.  
Another related result is a theorem of McAlister stating,
for an idempotent $e$ with image of size $n-1$ and for a permutation group
$G\leq\sym$, that the 
semigroup $\genset{G,e}$ is regular (see~\cite[Theorem~3.10]{mcalister}). 
%JDM precise reference required.

Theorems~\ref{main} and~\ref{th2} generalize   the results of Levi and McFadden \cite{lm,levi96} referred to above, and also provide a converse.

The class of semigroups of the form $\genset{G,a}\setminus G$ where 
$G\leq\sym$ and $a\in\trans$ is related to the class of 
$\sym$-normal semigroups introduced in \cite{sullivan}. %JDM reference required!
 For a semigroup $S\leq\trans$ define 
$$
\norm(S)=\{g\in\sym\ |\ S^g=S\},
$$
where $S^g=\{a^g\ |\ a\in S\}$. 
It is clear that each element of $\norm(S)$ induces an automorphism of $S$,
and this gives rise to a homomorphism from  $\norm(S)$ into $\aut S$. If the kernel 
of  this homomorphism is trivial, then we can consider $\norm(S)$ as a subgroup 
of $\aut S$. A semigroup $S\leq\trans$ is said to be {\em $\sym$-normal}
if $\norm(S)=\sym$. 
Schreier \cite{Sc36} and Mal'cev \cite{Ma52}, more or less explicitly,
proved, for a semigroup $S$ containing all the constant mappings,
that $\aut S=\norm(S)$. 
The class of $\sym$-normal semigroups was 
characterized in \cite{lm} as follows: 
a semigroup $S\leq \trans\setminus \sym$ is $\sym$-normal if and only if 
$$
S=\bigcup_{\alpha\in S} \genset{\sym,\alpha}\setminus\sym.
$$
This serves as a further incentive to study the class of semigroups 
$\genset{G,a}\setminus G$ where $G\leq\sym$ and $a\in\trans$. 

%JDM up to here? 

The proofs of our main theorems rely on techniques from the theory of 
permutation groups and on explicit machine computations. 
In Section~\ref{sect2} we introduce a new property for 
 permutation groups, namely the universal transversal property 
(see Definition~\ref{utdef}).   We refer to groups satisfying this property as universal transversal groups.
We show in Lemma~\ref{impliesutp}
that the assertions in Theorem~\ref{main}(i)--(ii) and 
in Theorem~\ref{th2}(ii)--(iii) 
 imply that $G$ satisfies the universal 
transversal property.
(It is not too hard to show that Theorem~\ref{th2}(ii) follows from
Theorem~\ref{th2}(i); see the proof of Theorem~\ref{th2} in Section~\ref{proofsec}.)
The remaining part of Theorem~\ref{main} is verified by 
first giving an explicit list of the permutation groups
that satisfy the universal transversal property (Theorem~\ref{superlemma}),
and then verifying one-by-one which of them gives rise to an 
idempotent generated semigroup.
The converse of Theorem~\ref{th2} is proved using the results 
of~\cite{lmm}.
The computations in the last stage of the proof were often carried
out using the computational algebra system {\sf GAP}~\cite{GAP}. Full details of 
these computations are available on the paper's companion webpage 
\cite{webpage}; see also Section~\ref{computsection}.

As mentioned in the previous paragraph, the proofs of the two main theorems
are partly based on the classification of the universal transversal groups, 
given in Theorem~\ref{superlemma}.
The action of a universal transversal group relates transversals 
and partitions in some way, and so the class of universal transversal groups
is closely related to the class of synchronizing groups. In fact it is shown in 
Lemma~\ref{unisynch} that a universal transversal group is synchronizing 
(see Section~\ref{classutg} for the definition).
The class of synchronizing groups has recently been a very active
research topic in the theory of finite permutation groups; 
see~\cite{ar}, \cite{steinberg}, \cite{cameron}, 
 and \cite{neumann}. 
As synchronizing groups are primitive, it follows that universal transversal groups are too.
It has been known for a long time that primitive subgroups
of $\sym$ that do not contain $\alt$ must be small in comparison to the size of 
$\sym$; the estimate which is most useful for us is given by 
Mar\'oti~\cite{maroti}. On the other hand, the 
nature of universal transversal groups imply that they must be large, and 
we use these two facts to classify them completely.

As Mar\'oti's theorem uses the classification of finite simple groups, our
main results also depend on the classification. Of course, the size
of a primitive permutation group can be bounded independently of the 
classification; see the introduction of~\cite{maroti} for references. However,
these  bounds were not sufficiently practical for the purposes 
of proving Theorems~\ref{main} and~\ref{th2}. 

The structure of the paper is as follows. In Section~\ref{sect2} we establish 
the connection between the classes of regular semigroups, idempotent
generated semigroups, and universal transversal groups. The classification of universal 
transversal groups is stated without proof 
in the same section (Theorem~\ref{superlemma}). Assuming that we know the
complete list of universal transversal groups, we prove the main theorems
of the paper in Section~\ref{proofsec}. The proof of Theorem~\ref{superlemma}
is given in Section~\ref{classutg}. Finally, in Section~\ref{computsection},
we describe in more detail the machine calculations that were used in the proofs
of our results.

%The system of notation we use in this paper is standard. %JDM this is not a helpful comment.
We refer the reader to \cite{Ho95} for further information on the fundamentals of semigroup theory.

%%%%%%%%%
%%%%%%%%%

\section{Semigroups and the universal transversal property}\label{sect2}

The main objective of this paper is to investigate the semigroups 
$\genset{G,a}\setminus G$ and $\genset{a^g\ |\ g\in G}$ where $G$ is a subgroup 
of $\sym$ and $a$ is a non-invertible transformation in $\trans$. Since,
for $a\in\trans\setminus\sym$ and $g\in \sym$, the element $a^g$ is not 
invertible, we obtain that $\genset{a^g\ |\ g\in G}\leq 
\genset{G,a}\setminus G$. 
The next result states that certain semigroups arising from alternating and symmetric groups are regular and idempotent generated.

\begin{lem}\label{antigo}
Let $a\in  \trans\setminus \sym$ and let $S=\genset{a^g\mid g\in \sym }$. Then the following hold:
\begin{enumerate}
\item[\rm (i)] $S=\genset{\sym, a}\setminus \sym=\genset{\alt, a}\setminus \alt=\genset{a^g\mid g\in \alt }$;
\item[\rm (ii)] $S$ is idempotent generated;
\item[\rm (iii)] $S$ is regular. 
\end{enumerate}
\end{lem}
\begin{proof}
It is easy to see, for $b\in S$ and $g\in\sym$, 
that $b^g\in S$, which shows that the semigroup $S$ is $\sym$-normal.
Thus assertions~(ii) and (iii) follow 
from~\cite[Theorem~6 and Proposition~9]{lm}. That 
$S=\genset{\sym,a}\setminus \sym$ is noted after~\cite[Proposition~4]{lm},
while  $S=
\genset{g^{-1}ag\ |\ g\in\alt}$ follows from~\cite[Proposition~6]{levi96}. 
Therefore
$$
S=\genset{\sym, a}\setminus \sym\geq \genset{\alt, a}\setminus \alt
\geq \genset{g^{-1}ag\ |\ g\in\alt}=S
$$
which gives that the two inequalities in the previous displayed line are,
in fact, equalities, and so~(i) holds.
\end{proof}

The following result is noticed  by McAlister in the proof
of ~\cite[Lemma~2.2]{mcalister}. 

\begin{lem}\label{mcal}
For $a \in \trans\setminus\sym$ and $G\leq \sym$, the semigroups $\genset{G, a}\setminus G$ and $\genset{a^g\mid g\in G}$ 
have the same set of idempotents. Consequently, 
if $\genset{G, a}\setminus G$ is idempotent generated,
then $\genset{G,a}\setminus G=\genset{a^g\mid g\in G}$.
\end{lem}
\begin{proof}
For a semigroup $S$, let $E(S)$ denote the set of idempotents in $S$.
Set $S_1=\genset{a^g\mid g\in G}$ and $S_2=\genset{G, a}\setminus G$.
As noted above, $S_1 \leq S_2$, and so
we are only required to  prove that $E(S_2 ) \subseteq E(S_1)$.
Every element of the semigroup 
$S_2$ can be written 
as $g_1a g_2a\cdots g_{n}a g_{n+1}$ where 
$g_i\in G$. 
Let $u=g_1a g_2a g_3 \ldots g_na g_{n+1}$ be an
idempotent of $S_2$. Then
\arraycolsep 2pt
\[
   u= a^{g_1^{-1}} a^{(g_1 g_2)^{-1}} a^{(g_1 g_2 g_3)^{-1}} \ldots
a^{(g_1 g_2 g_3\ldots g_n)^{-1}}(g_1 \ldots g_{n+1}).
\]
Write
$u=vg$, where $g = g_1 \ldots g_{n+1}$ and $v\in S_1$.
   Now,
as $G$ is a finite group,
there exists $n\geq 1$ such that $g^n$ is the identity and,  as
$u=v g$
is
idempotent, we have
\[
\begin{array}{rcl}
    v g &= & (v g)^n  \\
             &=&  v (g v g^{-1} )(g^2 v g^{-2})\ldots
(g^{n-1} v g^{-n+1})g^n\\
             &=&   v (g v g^{-1} )(g^2 v g^{-2})\ldots
(g^{n-1} v g^{-n+1}) \in S_1.
    \end{array}
\]
Hence $u=vg$ is an idempotent of $S_1$, and so $E(S_2 ) \subseteq E(S_1)$, 
as claimed. 

To prove the second assertion, suppose that $E$ denotes the set of idempotents in 
$S_2$ and assume that 
$\genset E=S_2$. 
The first assertion
of the lemma implies that $E\subseteq S_1$, and so $S_2=\genset{E}\leq S_1$.
As $S_1\leq S_2$, the equality $S_1= 
S_2$ follows.
\end{proof}

%\section{The universal transversal property}\label{statement}

Next we define the {\em universal transversal property} for 
permutation groups. Let $P$ be a partition of a set $X$. 
Recall that a subset $I$ of $X$ is called a 
\emph{transversal} for $P$ 
if every class of $P$ contains precisely one element of $I$.
If $X$ is a set and $a$ is a transformation of $X$, then the image
of an element $\alpha\in X$ under $a$ is denoted by $\alpha a$. If $Y$ is a subset of
$X$, then we may consider the image $Ya=\{\alpha a\ |\ \alpha\in Y\}$ of $Y$ under $a$.

\begin{de}\label{utdef}
A permutation group $G$ of degree $n$ is said to have the \emph{universal trans\-versal property} if for every subset $I$ of $\{1,2,\ldots, n\}$ and every partition $P$ of $\{1,2,\ldots, n\}$ with $|I|$ classes, there exists $g\in G$ 
such that $Ig$ is a transversal for $P$.
\end{de}

In the following examples we describe a group that satisfies the 
universal transversal property, and another that does not.
Recall that if $G$ is a permutation group acting on a set $X$, then, for
$\alpha\in X$, the set $\{\alpha g\ |\ g\in G\}$ is said to be a 
{\em $G$-orbit} and is denoted by $\alpha G$. 
The set of 
$G$-orbits form a partition of $X$ and $G$ is called 
{\em transitive} if $X$ is a $G$-orbit.
The permutation group $G$ acts 
on the set of subsets of $X$ (defined above) 
and we may consider the $G$-orbit of a subset
$Y\subseteq X$.

\begin{ex} Set $G=\left<(1\,2\,3\,4\,5)\right>$. Then $G$ is a cyclic group 
with order $5$, and 
we claim  that $G$ satisfies the universal transversal property. 
For $r\in\{1,4,5\}$, the group 
$G$ is transitive on the set of subsets containing precisely $r$-elements. Hence 
we only need to prove this claim for transversals and partitions containing 
either~two or~three members.
The $G$-orbits on the set of subsets with two elements are
\begin{eqnarray*}
\{1,2\}G&=&\{\{1,2\},\{2,3\},\{3,4\},\{4,5\},\{5,1\}\}\quad\mbox{and}\\
\{1,3\}G&=&\{\{1,3\},\{2,4\},\{3,5\},\{4,1\},\{5,2\}\}.
\end{eqnarray*}
Suppose that $P$ is a partition of $\{1,\ldots,5\}$ with two classes
such that none of the elements in one of the orbits is a transversal for $P$. 
Then every pair of elements in the orbit must lie in the same class of $P$, 
which implies that  $P$ contains only one class, a contradiction. 
The $G$-orbits on the set of subsets with three elements are 
$\{1,2,3\}G$ and $\{1,2,4\}G$. An argument similar to the one above, shows that
a partition that has no transversal in one of these orbits can have at 
most two classes. Hence $G$ must have the universal 
transversal property.
\end{ex}

\begin{ex}
Set $G=\left<(1\,2\,3\,4\,5\,6\,7)\right>$. Then
$G$ is a cyclic group with order 7.
In contrast with the previous example, 
$G$ does not satisfy the universal transversal property.  Indeed,  an
easy calculation shows that the $G$-orbit
$$ \{\{ 1, 2, 3 \}, \{ 2, 3, 4 \}, \{ 3, 4, 5 \}, \{ 4, 5, 6 \}, \{ 5, 6, 7 \},
  \{ 1, 6, 7 \}, \{ 1, 2, 7 \} \}$$
contains no transversal for the partition $\{ \{ 1 \}, \{ 2, 3, 4, 6, 7 \}, \{ 5 \} \}$.
\end{ex}

Recall that the \emph{kernel} of a transformation $a\in\trans$ is the 
equivalence relation $\set{(\alpha,\beta)}{\alpha a=\beta a}$ and is denoted by $\ker a$. The image of $a\in\trans$ is denoted by $\im a$. The rank of a transformation
$a$ is defined as $|\im a|$. 
Following~\cite{lmm}, for a given $G\leq\sym$ and $a\in\trans$, we set
$$
K_G(a)=\{g\in G\ |\ \rank a=\rank aga\}.
$$
It is straightforward to verify (and also noted before~\cite[Theorem~2.3]{lmm}) that % noted before~\cite[Theorem~2.3]{lmm} that 
$g\in K_G(a)$ if and only if $g$ maps $\im a$ into a transversal of $\ker a$.

\begin{lem}\label{kgalem}
Let $a\in\trans\setminus\sym$ and  let $G\leq\sym$. Then 
$G$ satisfies the universal transversal property if and only if 
 $K_G(a)\neq\emptyset$ for all $a\in\trans$.
\end{lem}
\begin{proof}
%Suppose first that $g\in K_G(a)$. Then $\rank aga=\rank a$, and so $(\im a)ga=\im a$. 
%Thus distinct elements of $(\im a) g$ must have distinct images under $a$, 
%which gives that
%$(\im a)g$ must be a transversal for $\ker a$. Reversing the argument proves that
%the converse also holds.
%To prove the second assertion, 
Note that $G$ satisfies the universal 
transversal property if and only if for all $a\in\trans$ there is some $g\in G$
such that $\im ag$ is a transversal for $\ker a$; that is $g\in K_G(a)$. 
Hence $G$ satisfies the universal transversal
property if and only if $K_G(a)\neq \emptyset$ for all $a\in \trans$. 
\end{proof}

The main results of this paper rely on the classification 
of universal transversal groups, given in the following theorem. 
The permutation groups that appear in this theorem are considered in 
their natural actions. The group $\agl(1,p)$ acts on the $p$ vectors 
of a one-dimensional vector space over the field of $p$~elements, and
the groups $C_5,\ \dihed 5$ are considered as subgroups of $\agl(1,5)$. 
The projective groups $\pgl$, $\psl$, and $\pgaml$ are viewed as
permutation groups acting on the set of projective points (that is,
the set of one-dimensional subspaces) of their
natural module.

\begin{thm} \label{superlemma}
A subgroup $G$ of $\sym$ has the universal
transversal property if and only if one of the following is valid:
\begin{enumerate}
    \item[\rm (i)] $n=5$ and $G\cong C_5,\ \dihed{5},$ or $\agl(1,5)$;
    \item[\rm (ii)] $n=6$ and $G\cong \psl(2,5)$ or $\pgl(2,5)$;
    \item[\rm (iii)]  $n=7$ and $G\cong\agl(1,7)$;
    \item[\rm (iv)] $n=8$ and $G\cong\pgl(2,7)$;
    \item[\rm (v)] $n=9$ and $G\cong\psl(2,8)$ or $\pgaml(2,8)$;
    \item[\rm (vi)] $G=\alt$ or $\sym$.
\end{enumerate}
\end{thm}

The proof
of Theorem~\ref{superlemma}  
will be given in Section~\ref{classutg}.

\section{The proofs of the main results}\label{proofsec}

In this section we prove Theorems~\ref{main} and~\ref{th2}.
We start with a lemma that
links the universal transversal property with the classes of 
idempotent generated semigroups and regular semigroups. This lemma essentially 
allows us to prove one direction of the main theorems.

\begin{lem}\label{impliesutp} 
Let $G$ be a subgroup of $\sym$ such that one of the following properties holds:
\begin{itemize}
\item[(i)] 
$\genset{G, a}\setminus G$ is idempotent generated
for all $a\in \trans\setminus \sym$;
\item[(ii)] $\genset{G, a}\setminus G$ is regular
for all $a\in \trans\setminus \sym$;
\item[(iii)]
$\genset{a^g\mid g\in G}$ is idempotent generated for all  $a\in \trans\setminus \sym$.
\item[(iv)] $\genset{a^g\mid g\in G}$ is regular for all  $a\in \trans\setminus \sym$. 
\end{itemize}
Then $G$ has the universal transversal property.
\end{lem}
\begin{proof}
Let us assume by contradiction that $G$ does not have the universal transversal
property, and show that assertions~(i)--(iv) are
not valid. By Lemma~\ref{kgalem}, there is some $a\in\trans$ such that 
$K_G(a)=\emptyset$. We
 note that $\genset{a^g\ |\ g\in G}\leq \genset{G,a}\setminus G$
and that $a\in \genset{a^g\ |\ g\in G}\cap \genset{G,a}\setminus G$. 

Let $b$ be an idempotent in $\genset{G, a}\setminus G$ and hence in $\genset{a^g\:|\:g\in G}$ by Lemma \ref{mcal}. Then $b$ is of the form
$g_1ag_2a\cdots g_{k-1}ag_k$ where $g_1,\ldots,g_k\in G$.  
We claim that $k\geq 3$. Indeed if $g_1ag_2$ is an idempotent with 
some $g_1,\ g_2\in G$, then $g_1ag_2g_1ag_2=g_1ag_2$ and so $ag_1g_2a=a$. 
In particular, $\rank ag_1g_2a=\rank a$, which is impossible as 
$K_G(a)=\emptyset$. Hence $b=g_1ag_2a\cdots g_{k-1}ag_k$ with $k\geq 3$, 
as claimed. Note that $\rank b\leq\rank ag_2a$ and,
as $g_2\not\in K_G(a)$, that $\rank ag_2a<\rank a$. 
Therefore $\rank b<\rank a$, and the element 
$a$ is not a member of  the semigroup generated by the idempotents
of $\genset{G,a}\setminus G$ or $\genset{a^g\:|\:g\in G}$.
Thus assertions~(i) and~(iii) do not hold.

To prove that assertions~(ii) and~(iv) fail, we show that $a$ is not a regular
element of $\genset{G, a}\setminus G$. 
Suppose as above that $a$ is regular and 
there is $b\in \genset{G, a}\setminus G$ such that $a=aba$. 
Then $\rank ab=\rank a$. In addition,
$ab=abab$, and so $ab$ is an idempotent.
We obtain a contradiction, as, by the previous paragraph, 
the semigroup 
$\genset{G, a}\setminus G$ has no idempotents with the same rank
as $a$. That is, $a$ is not a regular
element of $\genset{G, a}\setminus G$, and, since $\genset{a^g\:|\:g\in G}\leq \genset{G, a}\setminus G$,
it is not a regular element of  $\genset{a^g\mid g\in G}$. 
Therefore  
$\genset{G, a}\setminus G$ is not a regular semigroup,
and neither is its subsemigroup $\genset{a^g\mid g\in G}$.
\end{proof}

We can now prove  Theorem~\ref{main}. 
The symbol $[\alpha_1,\alpha_2,\ldots,\alpha_n]$ denotes the element of 
$\trans$ that maps $1\mapsto \alpha_1,\ 2\mapsto\alpha_2,\ldots,
n\mapsto\alpha_n$.

\begin{proof}[Proof of Theorem~$\ref{main}$]
If $\genset{G,a} \setminus G$ is idempotent generated,
then, by Lemma~\ref{mcal}, $\genset{a^g\mid g\in G}=\genset{G,a} \setminus G$, and so $\genset{a^g\mid g\in G}$ is also
idempotent generated.
Thus statement~(i) implies statement~(ii). 

Let us next show that statement~(ii) implies statement~(iii). 
By assumption, $\genset{a^g\ |\ g\in G}$ is idempotent generated for all
$a\in\trans\setminus\sym$. Hence Lemma~\ref{impliesutp} implies that  $G$ satisfies 
the universal transversal property. By Theorem~\ref{superlemma}, 
it suffices  to show, for $G\in\{C_5,\ \dihed{5},\ \agl(1,7),\ \pgl(2,7),\ \psl(2,8),\ \pgaml(2,8)\}$, that there exists some $a\in\trans\setminus\sym$
such that the semigroup $\genset{a^g\ |\ g\in G}$ is
not generated by idempotents. 
Using the {\sf GAP} computational algebra system, it is possible to show
that the required transformations $a$ are
$[1,3,2,2,2]$, $[1,2,3,3,3]$, $[1,2,3,3,3,3,3]$, 
$[6,2,3,4,6,6,6,6]$, $[1,2,3,5,4,5,4,4,5]$, and $[1,2,3,5,4,5,4,4,5]$,
respectively; see Section~\ref{computsection} and 
\cite{webpage} for further description of the computations.

The assertion that statement~(iii) implies statement~(i)
is verified as follows. The groups $\sym$ and $\alt$ satisfy statement~(i), by
Lemma~\ref{antigo}, 
so we may assume that $G$ is one of the groups 
$\agl(1,5)$, $\psl(2,5)$, $\pgl(2,5)$.
We used the computational algebra package {\sf GAP} to verify, 
for these groups, that for all $a\in\trans\setminus\sym$ the semigroup 
$\genset{G,\ a}\setminus G$ 
%$\genset{G,a}\setminus G$ 
is idempotent generated.  See Section~\ref{computsection} and~\cite{webpage}
for the details.\end{proof}

Let us now prove Theorem~\ref{th2}.

\begin{proof}[Proof of Theorem~$\ref{th2}$]
First we prove that~(i) implies~(ii).  It suffice to prove that if $a\in \trans\setminus \sym$ such that $\genset{G,a}$ is regular, then $\genset{G, a}\setminus G$ is regular. 
Assume that $\genset{G,a}$ is regular for some $G\leq\sym$ and 
for some $a\in\trans\setminus\sym$. 
Then for each $u \in \genset{G,a}\setminus G$ 
there is some $v\in \genset{G,a}$ with $u=uvu$.
If $v\not\in G$, then $u$ is a regular element of $\genset{G,a}\setminus
G$, and so we may assume that $v\in G$. Then
$vuv\not\in G$ and 
$u(vuv)u=(uvu)vu=uvu=u$. Thus $u$ is regular in $\genset{G,a}\setminus
G$ also in this case, and so $\genset{G,a}\setminus
G$ is a regular semigroup, as claimed.

Next we prove that assertion~(ii) implies (iii). Suppose that
$G$ is a permutation group such that $\genset{G,a}\setminus G$ is
regular for all $a\in\trans\setminus\sym$. Then, by Lemma~\ref{impliesutp}, 
$G$ satisfies the universal transversal property, and so $K_G(a)\neq\emptyset$
for all $a\in\trans$. Let $a\in\trans\setminus\sym$, and let $b\in\genset{a^g\ | g\in G}$. Then 
$K_G(b)\neq\emptyset$, and hence~\cite[Theorem~2.3]{lmm} implies
that $b$ is regular in $\genset{b^g\ |\ g\in G}$. As 
$\genset{b^g\ |\ g\in G}\leq  \genset{a^g\ |\ g\in G}$, we find that
$b$ is regular in $\genset{a^g\ |\ g\in G}$, and so $\genset{a^g\ |\ g\in G}$ is a regular semigroup.

The fact that~(iii) implies~(iv) is a consequence of Lemma~\ref{impliesutp} and Theorem~\ref{superlemma}.

Finally, we prove that~(iv) implies~(i). Let $a\in\trans\setminus\sym$. 
We are required
to show, for $b\in \genset{G,a}$, that $b$ is regular in $\genset{G,a}$. 
As $G$ has the universal
transversal property, Lemma~\ref{kgalem} gives that $K_G(b)\neq\emptyset$. 
Thus, by~\cite[Theorem~2.3]{lmm}, the element $b$ is regular in $\genset{G,b}$. 
As $\genset{G,b}\leq\genset{G,a}$, we find that $b$ is regular in 
$\genset{G,a}$, as claimed.
\end{proof}

\section{The classification of Universal transversal groups}\label{classutg}

The proof of Theorem~\ref{superlemma} is given in this section.
It can be verified using the computational algebra system 
{\sf GAP}~\cite{GAP}
that the groups 
$C_5$, $\dihed{5}$, $\agl(1,5)$, $\psl(2,5)$, $\pgl(2,5)$, $\agl(1,7)$, $\pgl(2,7)$, $\psl(2,8)$,
$\pgaml(2,8)$
listed in the theorem
satisfy the universal transversal property;
 full details of the computation are available on the companion 
webpage~\cite{webpage} (see also Section~\ref{computsection}). A permutation group  acting on a set $X$ is said to be {\em $k$-homogeneous}
if it is transitive on the set of subsets of $X$ with size $k$. 
The alternating group $\alt$ is $k$-homogeneous for $k=1,2, 
\ldots, n$; see \cite[Exercise 2.1.4]{dixon}. 
%JDM check the reference.
Thus the $k$-homogeneous groups 
$\alt$ and  $\sym$ have the universal transversal property, and 
so the proof of one direction  of Theorem \ref{superlemma} is concluded.

Synchronizing groups were first introduced in~\cite{steinberg} and 
a combinatorial characterization was given in~\cite{ar}. 
A permutation group $G\leq S_n$ is said to be {\em synchronizing} if for every non-trivial partition $P$ (a partition is said to be {\em non-trivial} if 
it has at least $2$ and at most $n-1$ blocks) 
of $\{1,\ldots ,n\}$ and every transversal $S$ of $P$, there exists $g\in G$ such that $Sg$ is not a transversal for $P$.  Recall that a permutation group is said to be {\em primitive} if
no non-trivial 
partition of the underlying set is invariant under the group action; 
see also~\cite[Section~1.5]{dixon}. 
If $G$ preserves a partition $P$ 
of $\{1,\ldots,n\}$,  then any image of a transversal of $P$ is again a 
transversal.
Hence a synchronizing group is transitive and  primitive 
(see also~\cite[Introduction]{neumann}). 
\begin{lem}\label{unisynch}
Every universal transversal group is synchronizing, and hence 
such a group is
transitive and primitive.
\end{lem}
\begin{proof}
Suppose that $G$ is a universal transversal group acting on a set $X$.
Let $P$ be a non-trivial partition of $X$ and let 
$S$ be a transversal for $P$. 
Let $\alpha,\ \beta\in X$ be two distinct elements in the same 
block of $P$.
Let $P_1$ be a partition of $X$ such that $\{\alpha\},\ \{\beta\}\in P_1$
and $|P|=|P_1|$. 
Now, by assumption, there exists $g\in G$ such that $Sg$ is a transversal for $P_1$ and hence $\alpha,\ \beta\in Sg$. Thus $Sg$ is not a transversal for $P$, which shows
that $G$ is synchronizing. Therefore
$G$ is transitive
and  primitive. 
\end{proof}

We note that it is possible to prove that a universal transversal group is 
transitive and primitive without using the concept of synchronizing groups.
However, we decided to include the proof above, as synchronizing groups will
play some further role in this paper.
%%%%%%%%%%%%
%%%%%%%%%%%%

A subgroup of $\sym$ is said to be {\em proper} if it does not contain 
$\alt$. 
Next we prove that a proper primitive group of large enough degree does not
satisfy the universal transversal property. 
Before proving this result 
we introduce some terminology. 
Let $r\in\mathbb{N}$. 
We say that a partition $P$ of a set is 
$r$-{\em singular} if it has $r+1$ classes, $r$ of which contain exactly $1$ element.
The number of $r$-singular partitions of $\{1,\ldots,n\}$ is 
clearly $\binom nr$. Further, if $R$ is a subset of $\{1,\ldots,n\}$
with $r+1$ elements, then the number of $r$-singular partitions
$P$ of $\{1,\ldots,n\}$ such that $R$ is a transversal for $P$ is $r+1$.

\begin{lem}\label{properprim} If $G$ is a proper primitive group with degree at least $47$, then $G$ does not satisfy the universal transversal property.
\end{lem}
\begin{proof}
It is known that proper primitive groups have `small' orders.
More precisely, in \cite[Corollary 1.1]{maroti} it is proved that if $G\leq
\sym$ is a proper primitive group, then
\begin{equation}\label{maroti}
|G|<50n^{\sqrt{n}}.
\end{equation}
Thus the strategy of the proof is to argue that universal transversal groups
must have order at least $50n^{\sqrt{n}}$. The first step is to find a lower bound for the order of a group with the universal transversal property.

 Suppose $R$ is an $(r+1)$-element set. Then the orbit of $R$ under $G$ must contain a transversal for every $r$-singular partition of $\{1,\ldots ,n\}$.  
As noted above, 
the number of $r$-singular partitions of $\{1,...,n\}$ is $\binom{n}{r}$ 
and each $(r+1)$-element set is a transversal for exactly $(r+1)$ 
$r$-singular partitions of $\{1,\ldots ,n\}$. Hence if $G$ has the 
universal transversal property, 
then 
\begin{equation}\label{utpineq}
|G|(r+1)\geq \binom{n}{r}
  \end{equation}
must hold for all $r\in\{1,\ldots,n-1\}$, and hence
\begin{equation}\label{eq1}
\binom nr\leq 50n^{\sqrt{n}}(r+1).
\end{equation}
We complete the proof by showing that~\eqref{eq1} fails to hold for $n\geq 47$ 
with $r=n/2$
when $n$ is even and $r=(n+1)/2$ when $n$ is odd.
That is, we show that
\begin{equation}\label{eq0.1}
50(2r)^{\sqrt{2r}}(r+1)<\binom{2r}{r}
\end{equation}
and that
\begin{equation}\label{eq0.2}
50(2r+1)^{\sqrt{2r+1}}(r+1)<\binom{2r+1}{r}
\end{equation}
holds for sufficiently large $r$. 
We only prove~\eqref{eq0.1}, as the proof of~\eqref{eq0.2} is very similar.
Let $A(r)$ and $B(r)$ denote the left-hand side and 
the right-hand side of~\eqref{eq0.1}, respectively. 
We use induction on $r$. It is easy to
compute that $A(r)<B(r)$ for 
$r\in \{24,\ldots,46
\}$. Assume that $r\geq 46$, that $A(r)<B(r)$, and let us show that
$A(r+1)<B(r+1)$.  We claim that $A(r+1)/A(r)<B(r+1)/B(r)$ which implies
immediately that $A(r+1)<B(r+1)$. First, note that 
$B(r+1)/B(r)=(2r+1)(2r+2)/(r+1)^2=2(2r+1)/(r+1)$, and that
$$
\frac{A(r+1)}{A(r)}=\frac{50(2r+2)^{\sqrt{2r+2}}(r+2)}{50(2r)^{\sqrt{2r}}(r+1)}
\leq \frac{2(2r+2)^{\sqrt{2r+2}}}{(2r)^{\sqrt{2r}}}.
$$
Hence it suffices to show  that 
\begin{equation}\label{eq2}
\frac{(2r+2)^{\sqrt{2r+2}}}{(2r)^{\sqrt{2r}}}<\frac{2r+1}{r+1}.
\end{equation}
It is easy to see that~\eqref{eq2} holds for $r=46$. The fact that it holds 
for $r\geq 47$ follows from the observation that the left-hand side 
of~\eqref{eq2} is decreasing, while the right-hand side is increasing. 
The latter of these claims is trivial. For the former, the derivative of the
function $x\mapsto (2x+2)^{\sqrt{2x+2}}/{2x^{\sqrt{2x}}}$ is 
\begin{multline*}
x\mapsto
C(x)\left(\frac{\sqrt x\sqrt{x+1}\log(x+1)+\sqrt{x}\sqrt{x+1}(\log 2+2)}{x^{(2\sqrt{2}\sqrt{x}+1)/2}(\sqrt{2}x+\sqrt{2})2^{\sqrt{2}\sqrt{x}}}\right.\\
\left.-\frac{(x+1)\log x+(\log 2+2)(x+1)}{x^{(2\sqrt{2}\sqrt{x}+1)/2}(\sqrt{2}x+\sqrt{2})2^{\sqrt{2}\sqrt{x}}}\right),
\end{multline*}
where $C(x)=(x+1)^{\sqrt 2\sqrt{x+1}}2^{\sqrt 2\sqrt{x+1}}$. 
Since $\sqrt x\sqrt{x+1}\log(x+1)<(x+1)\log x$
and $\sqrt x\sqrt{x+1}<(x+1)$, we obtain that this derivative
is negative, which shows that the left-hand side 
of~\eqref{eq2} is decreasing, as claimed.
\end{proof}

As the previous lemma gives a practical upper  bound for the degree of a
proper universal transversal group, we could  finish the classification
of such groups using computer calculation only. However, using the structure 
theorem of finite primitive permutation groups and some elementary 
combinatorics, we can significantly 
reduce the amount of computer calculation that is 
required to prove Theorem~\ref{superlemma}.
Primitive 
permutation groups are described by the O'Nan-Scott Theorem that divides these
group into several classes. Statements of this theorem can be found in~\cite[Section~4.8]{dixon} and in~\cite[Sections~4.4--4.5]{cam}, while 
in~\cite{bpsch} there is a detailed comparison of the different 
versions of the theorem that can be found in the literature.
Since the order of a non-abelian finite simple group is at least~60,
combining the O'Nan-Scott Theorem with the bound in Lemma~\ref{properprim}
gives that a proper universal transversal group is either an almost 
simple group, an affine group, or a 
subgroup of a wreath product in product
action. 

A finite group is said to be {\em almost simple} 
if it has a unique minimal
normal subgroup which is a non-abelian simple group. Almost simple
primitive groups form a class of primitive groups in the O'Nan-Scott Theorem.
If $G$ is a permutation group acting on 
$X$, then the wreath product $W=G\wr\sym$ can be considered as a 
permutation group acting on the cartesian product $X^n$. 
A primitive group of product action type is a suitable 
subgroup of such a wreath
product $W$ in the case when $G$ is an almost simple primitive group.

A primitive permutation group $G$ is said to be {\em affine} if it has
an abelian normal subgroup.
Affine primitive groups can be characterized as follows. Let $V$ be an 
$n$-dimensional vector space over a field $\mathbb F_p$ of $p$ elements for 
some prime $p$, and let $H$ be a subgroup of the group $\mbox{GL}(V)$ of
invertible linear transformations of $V$. Every element $v\in V$ defines a 
permutation $\tau_v$ of $V$ where $\tau_v:u\mapsto u+v$ for all $u\in V$. The collection $T$
of the $\tau_v$ is a subgroup of the full symmetric group
$\mbox{Sym}\,V$ isomorphic to the 
additive group of $V$. Similarly, the elements of $H$ can be considered
as permutations of $V$, and so $H$ can also be viewed as a subgroup of 
$\mbox{Sym}\,V$. It is easy to see that $H$ normalizes $T$, and so 
their product $TH$ is a subgroup of $\mbox{Sym}\,V$. 
As $T$ is transitive, so is $TH$. Further,
$TH$ is primitive if and only if $H$ is irreducible; that is, no 
non-trivial, proper subspace of $V$ is invariant under $H$. 
In this case, as $T$ is an abelian normal subgroup of $TH$, the primitive group $TH$
is affine.
Conversely, by~\cite[Theorem~4.7A]{dixon}, every affine primitive group
is permutationally isomorphic to a group of the form $TH$.
Permutational isomorphism is defined in~\cite[page~17]{dixon}. 
Permutationally isomorphic groups are essentially the same except for the 
labeling of the points on which they act.
In particular, the degree of an affine primitive group is a prime-power.
For each prime-power $p^k$ there is a largest affine group with
degree $p^k$ constructed as
follows. Let $V$ be the $k$-dimensional vector space over $\mathbb F_p$ and let
$T$ denote the subgroup formed by the $\tau_v$ as defined above. 
Then the group $T\mbox{GL}(V)$ is called the {\em affine general 
linear group} and is denoted by $\mbox{AGL}(k,p)$. Every affine
primitive group with degree $p^k$ is a subgroup of $\mbox{AGL}(k,p)$.

\begin{lem}\label{utpasth}
If $G$ is a permutation group of degree $n$ with the universal transversal property, then one of the following
must hold:
\begin{enumerate}
\item[(i)] $n\leq 4$;
\item[(ii)] $G\leq \mbox{AGL}(1,p)$ with $p\in\{5,\ 7\}$;
\item[(iii)] $G$ is an almost simple 
primitive group.
\end{enumerate}
\end{lem}
\begin{proof}
By Lemmas~\ref{unisynch} and~\ref{properprim}, 
a group with the universal  transversal property is
synchronizing, and hence it is primitive of degree at most~46.
As noted above, a primitive group of such  small degree is either
almost simple, or a subgroup of a wreath product, or affine.
As subgroups of wreath products in  product action are non-synchronizing (see~\cite[Example~3.4]{neumann}), 
they cannot have the universal transversal property, by Lemma~\ref{unisynch}. 
Since an affine primitive group  of degree $p^k$ is a subgroup of 
$\mbox{AGL}(k,p)$, and there is no affine group of degree~6, the statement
of the theorem is valid whenever $n\leq 7$. 
Thus we may assume without loss of generality that $n> 7$ 
and we are required to show that an affine 
primitive group  with degree~$n$ does not have the 
universal transversal property.

Assume that $G$ is a primitive group of affine type acting on a vector space
$V=\mathbb F_p^k$ for some prime $p$ and integer $k$. A translate of a 
one-dimensional subspace in $V$ is said to be a {\em line} and a translate
of a 2-dimensional subspace is said to be a {\em plane}. The action of 
$G$ on $V$ preserves the set of lines and the set of planes.
Suppose first 
that $p$ is at least~$5$ and $\dim V\geq 2$. 
In this case a line has at least 5 elements and there are three vectors which
do not lie on a common line. Choose 3 vectors $v_1$, $v_2$, $v_3$ 
which do not lie on a common line
and let $P$ be the partition 
$$
\{\{v_1\},\{v_2\},\{v_3\},V\setminus\{v_1,v_2,v_3\}\}.
$$
Now let $u_1$, $u_2$, $u_3$, $u_4$ be 4 vectors on a common line and set 
$$
S=\{ u_1,\ u_2,\ u_3,\ u_4\}.
$$
If $g$ is an element of $G$, then the image $Sg$ contains 4 vectors which lie
on a common line. On the other hand, if $R$ is a transversal for $P$, then $R$ must
contain $v_1$, $v_2$, $v_3$ and these do not lie on a common line. Thus 
no image $Sg$ of $S$ can be a transversal of $P$, and so $G$ does not have the
universal transversal property.

Suppose now that either $p=2$ and $\dim V\geq 4$ or $p=3$ and $\dim V\geq 3$. 
In this 
case 
we choose four vectors $v_1$, $v_2$, $v_3$, and $v_4$ 
on a common plane
and choose 5 vectors 
$u_1$, $u_2$, $u_3$, $u_4$, $u_5$ so that no four of them lie on a 
common plane. If $p=3$, then we may choose $u_1=0$, $u_2=b_1$, $u_3=b_2$, $u_4=b_3$, 
$u_5=b_1+b_2+b_3$, 
while if $p=2$ then $u_1=0$, $u_2=b_1$, $u_3=b_2$, $u_4=b_3$ and $u_5=b_4$
are suitable where $b_1,b_2,\ldots,b_k$ is a basis of $V$.
Then an argument similar to the one 
above shows 
that no image of $\{u_1,\ u_2,\ u_3,\ u_4,\ u_5\}$ by any element of $G$ is a transversal for the partition
$$
\{\{v_1\},\{v_2\},\{v_3\},\{v_4\},V\setminus\{v_1,v_2,v_3,v_4\}\}.
$$

%% Let us now  investigate the case $p=3$ and $\dim V=2$. Set 
%% $$
%% S=\{(0,0),(1,0),(2,0)\} 
%% $$
%% and 
%% $$
%% P=\{\{(0,0)\},\{(1,0),(2,0),(1,1),(2,2)\},\{(0,1),(0,2),(1,2),(2,1)\}\}.
%% $$
%% Note that $S$ is a line of $V$ and so are all of the images of $S$. 
%% If $Sg$ is a transversal for $P$ with some $g\in G$ then $(0,0)\in Sg$ and 
%% so $Sg$ is a one-dimensional subspace. 
%% On the other hand, the partition $P$ 
%% satisfies the following property: if $v$ is a non-zero
%% element of a block $B$ then $B$ contains another non-zero 
%% element of the one-dimensional subspace spanned by $v$. This shows that
%% no image $Sg$ of $S$ is a transversal for $P$. 

Assume now that $\dim V=1$ and let $p\geq 11$. We identify $V$ with 
$\mathbb F_p$. Set 
$$
P=\{\{0\},\{1\},\{2\},\mathbb F_p\setminus\{0,1,2\}\}
$$
and 
$$
S=\{0,1,3,4\}.
$$
If $g\in G$, then there are $a,\ b\in\mathbb F_p$ such that $xg=(x+a)b=xb+ab$
for all $x\in \mathbb F_p$. Thus $0g=ab$, $1g=b+ab$, $3g=3b+ab$, $4g=4b+ab$. 
If $Sg=\{ab,b+ab,3b+ab,4b+ab\}$ is a transversal for $P$, then there are three 
elements $x$, $y$, and $z$ of $Sg$ such that $x-y=y-z$. However, using that $p\geq 11$, inspecting all 24 possibilities for $x$, $y$, and $z$, 
we find that it is not possible to choose such elements.
  
The remaining cases $p=2,\ \dim V=3$  and $p=3,\ \dim V=2$ can be handled
as follows. Let $G\leq\mbox{AGL}(3,2)$, then $G$ can be viewed as a 
group acting on the 3-dimensional vector space $V$ over $\mathbb F_2$. 
It is not hard to verify that the orbit of the subset
$$
\{(0,0,0),(0,0,1),(0,1,0),(0,1,1)\}
$$
does not contain a transversal for the partition
$$
\{\{(0,0,0),(0,0,1)\},\{(0,1,0),(0,1,1),(1,0,0),(1,0,1)\},\{(1,1,0)\},\{(1,1,1)\}\}.
$$
The last case is $G\leq\mbox{AGL}(2,3)$ and hand computation can show that
the orbit of the subset
$$
\{(0,0),(0,1),(0,2)\}
$$
does not contain a transversal for the partition
$$
\{\{(0,0)\},\{(0,1),(0,2),(1,2),(2,1)\},\{(1,0),(2,0),(1,1),(2,2)\}\}.
$$
The proof is now complete. \end{proof}

Now we are ready to prove Theorem~\ref{superlemma}.

{\footnotesize
\begin{table}
\begin{tabular}[f]{|l|l|l|l|l|}
\hline
Deg&Group(s)&{\sf GAP} id.&Set&Partition\\
\hline
$7$&$\psl(3,2)$&5&$\{1,2,4\}$&$\{1\},\{2,3,4,7\},\{5,6\}$\\
& $7:3$ & 3 & $\{1,2,4,7\}$&$\{1\},\{2\},\{3\},\{4,5,6,7\}$\\
&$\dihed{7}$&2&$\{1,3,7\}$&$\{1\},\{2\},\{3,4,5,6,7\}$\\
\hline
$8$&$\psl(2, 7)$&4&$ \{1,2,3,5\}$&$\{1\},\{2\},\{3,4,5,7\},\{6,8\}$\\
\hline
$10$&$\mathcal{A}_{5}$,  $\mathcal{S}_{5}$ & 1, 2 &$\{1,2,3,5,6\}$&$\{1\},\dots,\{5\},\{6,\ldots,10\}$\\
 &  $\psl(2, 9),$ $\pgl(2, 9)$ & 3, 4 & & \\
  & $\mathcal{S}_{6}$,  $\M(10)$ & 5, 6 & & \\
 & $	\pgaml(2, 9)$ & 7 && \\
\hline
$11$&$\psl(2, 11)$&5&$\{ 1, 2, 3, 5 \}$&$\{1\},\{2\},\{3,\ldots,11\},\{9\}$\\
&$ \M(11) $&6&$\{1,2,3,4,6\}$&$\{1\},\{2\},\{3\},\{4,5,6,$\\
&&&& \hfill$7,10,11\}, \{8,9\}$\\
\hline
$12$&$\M(12)$&2&$\{1,\ldots,6\}$&$\{1,\ldots,6\},\{7,8\},\{9\},\ldots,$\\
&&&&\hfill$\{12\}$\\
&$\M(11)$&1&$\{1,2,3,4,11,12\}$&$\{1\},\ldots,\{5\},\{6,\ldots,12\}$\\
&$\pgl(2, 11),\psl(2, 11)$&4, 3&$ \{1,2,3,4,6,7\}$&$\{1\},\ldots ,\{5\},\{6,\ldots,12\}$\\
\hline
$13$&$\psl(3, 3)$& 7 &$\{1,\ldots,5,7,8\}$& $\{1\},\ldots, \{6\},\{7,\ldots,13\}$\\
\hline
$14$&$\pgl(2,13),\psl(2,13)$&2, 1 &$\{1,\ldots,6,9,12\}$&$\{1\},\ldots,\{7\},\{8,\ldots,14\}$\\
\hline
$15$&$\psl(4, 2)$, $\mathcal{A}_{7}$&4, 1&$\{1,\ldots,6,8,12\}$&$\{1\},\ldots,\{7\},\{8,\ldots,15\}$\\
\hline
$17$&$\psl(2, 2^4)$, $\psl(2, 2^4):4$&6,8&$\{1,\ldots,7,11,14\}$&$\{1\},\ldots,\{8\},\{9,\ldots,17\}$\\
&$\psl(2, 2^4):2$&7&$\{8,\ldots,14,16,17\}$&$\{1\},\ldots,\{8\},\{9,\ldots,17\}$\\
\hline
$18$&$\pgl(2,17)$&2&$\{1,\ldots,8,10,11\}$&$\{1\},\ldots,\{9\}, \{10,\ldots,18\}$\\
\hline
$21$&$\psigl(3, 4)$, $\pgl(3, 4)$,&5, 6 &$\{1,\ldots,9,11,13\}$&$\{1\},\ldots, \{10\}, \{11,\ldots,21\}$\\
&$\pgaml(3, 4)$ & 7& & \\
\hline
$22$&$\M(22)$, $\M(22):2$&1, 2&$\{1,\ldots,10,12,15\}$&$\{1\},\ldots, \{11\},\{12,\ldots,22\}$\\
\hline
$23$&$\M(23)$&5&$\{1,\ldots,5,8,11\}$&$\{1\},\ldots,\{6\},\{7,\ldots, 23\}$\\
\hline
\end{tabular}\vspace{\baselineskip}
\caption{}\label{table2}
\end{table}}

\begin{proof}[The proof of Theorem~$\ref{superlemma}$] 
Suppose that $G\leq\sym$ is a proper primitive permutation group that satisfies
the universal transversal property. We may assume that $G$ is not listed
in Theorem~\ref{superlemma}. 
A primitive group of degree~3 or~4 is either an alternating 
or a symmetric group. A primitive group of degree 5, is either 
an alternating or symmetric group, or listed in 
Theorem~\ref{superlemma}(i). Thus $G$ must either be a subgroup of 
$\mbox{AGL}(1,7)$  or an almost simple group. 
Further, we proved in Lemma~\ref{properprim}, 
that the order of $G$ must satisfy~\eqref{utpineq} with $r=\lfloor n/2\rfloor$. 
Using the primitive groups library of {\sf GAP}, we found that there are 
32 such groups and they are listed in the second column of 
Table~\ref{table2}. The third column of the table contains the catalogue 
number of the groups in the primitive groups library of {\sf GAP}. For instance,
the group $\mbox{PSL}(3,2)$ can be accessed as {\tt PrimitiveGroup( 7, 5 )}.
%% \begin{table}
%% $$
%% \begin{array}{|l|l|}
%% \hline
%% \mbox{degree} & \mbox{groups} \\
%% \hline
%% 7 &  \psl(3, 2) \\
%% \hline
%% 8 & \psl(2, 7)\\
%% \hline
%% 9 & \M(9) \\
%% \hline
%% 10 & \mathcal{A}_{5}, \mathcal{S}_{5}, \psl(2, 9), \pgl(2, 9), \mathcal{S}_{6}, \M(10), \pgaml(2, 9)\\
%% \hline
%% 11 & 
%%   \psl(2, 11), \M(11)\\
%% \hline
%% 12 & \M(11), \M(12), \psl(2, 11), \pgl(2, 11)\\
%% \hline
%% 13 &   \psl(3, 3)\\
%% \hline
%% 14 &  \psl(2,13), \pgl(2,13)\\ 
%% \hline
%% 15 & \mathcal{A}_{7}, \psl(4, 2)\\
%% \hline
%% 17 & \psl(2, 16), \psl(2, 16):2, \psl(2, 16):4\\  
%% \hline
%% 18 &  \pgl(2,17)\\
%% \hline
%% 21 &  \psigl(3, 4), \pgl(3, 4), \pgaml(3, 4)\\
%% \hline
%% 22 &\M(22), \M(22):2\\
%% \hline
%% 23 & \M(23)\\
%% \hline
%% 24 &  \M(24)\\
%% \hline
%% \end{array}
%% $$
%% \caption{}\label{possgroups}
%% \end{table}
In order to complete the proof of Theorem \ref{superlemma} 
it is necessary to find a partition $P$ and a subset $S$ of 
$\{1,2,\ldots, n\}$ such that no element in $\orb SG$ is a transversal for $P$. 
Such subsets and partitions can be found in Table~\ref{table2}. 
\end{proof}
%% For example, in $L(2, 11)$ the orbit of $\{1,2,3,4,7,8\}$ is
%% \begin{multline*}
%%  \{\{1, 2, 3, 4, 7, 8 \}, \{ 2, 4, 5, 8, 10, 11 \}, \{ 1, 2, 7, 9, 10, 11 \},
%%   \{ 2, 3, 5, 6, 7, 10 \}, \\\{ 3, 4, 5, 7, 9, 11 \}, \{ 1, 2, 4, 5, 6, 9 \},
%%   \{ 1, 3, 4, 6, 10, 11 \}, \{ 2, 3, 6, 8, 9, 11 \}, \\\{ 1, 5, 6, 7, 8, 11 \},
%%   \{ 4, 6, 7, 8, 9, 10 \}, \{ 1, 3, 5, 8, 9, 10 \} \}
%% \end{multline*}
%% and there are no transversals of the partition $\{\{1\},\{2\},\{3\},\{4\},\{5\},\{6,\ldots,$ $11\}\}$ in the orbit. Thus $L(2,11)$ does not have the universal transversal property.

%%%%%%%%%
%%%%%%%%%

\section{Computations}\label{computsection}

In this section we give a brief description of the methods used to perform the various computations  which are used above. 
More specifically, the following were verified computationally using {\sf GAP}~\cite{GAP}.
\begin{enumerate}
\item[\text (a)] In Theorem~\ref{main}, assertion (ii) implies assertion (iii): 
if $G$ is one of the groups  $C_5$, $D_5$, $\agl(1,7)$, $\pgl(2,7)$, $\psl(2,8)$, $\pgaml(2,8)$, then there exists an $a\in \trans\setminus \sym$ such that 
$\genset{G,a} \setminus G$ is not idempotent generated.
\item[\text (b)] In Theorem~\ref{main}, assertion~(iii) implies assertion~(i): if $G$ is one of the groups $\agl(1,5)$, $\psl(2,5)$, $\pgl(2,5)$, then $\genset{G,a} \setminus G$ is idempotent generated for all $a\in \trans\setminus \sym$.
\item[\text (c)] One direction in the proof of Theorem~\ref{superlemma}: if $G$ is any of the groups listed in Theorem 2.7 excluding $\alt$ and $\sym$, then $G$ has the universal transversal property.
\item[\text (d)] Table~\ref{table2}: if $G$ is any of the groups listed in Table 1, then $G$ does not have the universal transversal property. 
\end{enumerate}
We prepared a companion webpage~\cite{webpage} that contains full details 
of these computations. The procedures that were used in these computations
were collected into a~{\sf GAP} package, so that the reader can easily 
reproduce these computations herself. In addition to the {\sf GAP} package, 
the webpage contains detailed log files of the  computations that were necessary
to verify the statements above.

Parts (c) and (d) can be verified using {\sf GAP} by performing a brute force search. More precisely, if $G$ is a permutation group of degree $n$, then for every subset $I$ of $\{1,2,\ldots, n\}$ and  for every partition $P$ of $\{1,2,\ldots, n\}$ with $|I|$ classes, we verified that the orbit of $I$ under $G$ contained a transversal of $P$.  These computations are feasible due to the small degrees of the groups under consideration and the efficient methods in {\sf GAP} for computing with permutation groups. 

The idempotents with a specific rank in $\genset{G, a}\setminus G$  can be found using the simple orbit algorithm described in  \cite{Linton1998aa} and \cite{Linton2002aa}. Similarly simple orbit calculations, described in the same papers,  can be used to test membership in transformation semigroups. As such, the condition of the next lemma, equivalent to  $\genset{G, a}\setminus G$ being idempotent generated, can be verified efficiently using {\sf GAP}.

\begin{lem}\label{reduce2}
Let   $a\in\trans\setminus \sym$ and let $E$ denote the set of idempotents 
of $\genset{G, a}\setminus G$ with rank equal to $\rank a$. Then $\genset{G, a}\setminus G$ is idempotent generated if and only if every element
 of $Ga G$ is contained in $\genset{E}$.
\end{lem}
\begin{proof}
Let $S=\genset{G, a}\setminus G$. Then
  every element in $S$ can be given as a product $g_1a g_2a\cdots g_{n
-1}a g_n$ for some $g_1, g_2, \ldots, g_n\in G$. In particular, the set $GaG$ is a generating set for $S$. Thus $\genset{G, a}\setminus G$ is idempotent generated  if and only if every element of $Ga G$ lies in the subsemigroup generated by the idempotents of $S$. Every element in $Ga G$ has rank equal to that of $a$. It follows that $S$ is idempotent generated if and
 only if every element of $Ga G$ lies in $\genset{E}$.   
\end{proof}

%%%%%%%%%%%%%%%%

If $G\leq\sym$ and $\Gamma$ is a subset of $\{1,\ldots,n\}$, then  
the \emph{setwise stabilizer}
of $\Gamma$ in $G$ is the subgroup 
$G_\Gamma=\set{g\in G}{\Gamma g=\Gamma}$.  
The setwise stabilizer induces a group of permutations on the set 
$\Gamma$  denoted  by  $(G_\Gamma)^\Gamma$.  

\begin{lem}\label{reduce}
Let $a,\ b\in \trans$ where $\rank b^2=\rank b$ and such that there 
exist $g,\ h\in G$ with $(\ker a)g=\ker b$,  $(\im a)h
=\im b$, and $(g^{-1}a h)|_{\im b} (G_{\im b})^{\im b}=b|_{\im b} (G_{\im b})^{\im b}$. Then $\genset{G, a}=\genset{G, b}$.
\end{lem}
\begin{proof}
First we verify that the transformation $g^{-1}ah$ leaves $\im b$ invariant,  
and so the expression $(g^{-1}a h)|_{\im b}(G_{\im b})^{\im b}$ makes sense. Indeed, set $I=\im b$. 
The condition $\rank b^2=\rank b$ implies that $I$ is a transversal for
$\ker b$, and so $Ig^{-1}$ is a transversal for $(\ker b) g^{-1}=\ker a$. Hence
 $Ig^{-1}ah=(\im a)h=\im b$. Therefore $(\im b)g^{-1}ah=\im b$ as claimed.

Next we show that the conditions of the lemma imply that 
there is an element $u\in G_{\im b}$ such that $g^{-1}ahu=b$. 
As $(g^{-1}a h)|_{\im b} (G_{\im b})^{\im b}=b|_{\im b}
(G_{\im b})^{\im b}$, we obtain that there is an element $u\in G_{\im b}$
such that $g^{-1}ahu$ induces the same permutation on $\im b$ as $b$.  Let 
$\alpha\in\{1,\ldots,n\}$. Since $I$ is a transversal for $\ker b$, there is
an element $\overline\alpha\in I$ such that $\alpha$ and $\overline\alpha$ are
in the same block of $\ker b$; that is $\alpha b=\overline \alpha b$. 
This implies that $\alpha g^{-1}$ and $\overline\alpha g^{-1}$ are in the same
block of $\ker a$, and so $\alpha g^{-1}a=\overline\alpha g^{-1}a$. 
Thus
$$
\alpha g^{-1}ahu=\overline\alpha g^{-1}ahu=\overline\alpha b=\alpha b.
$$
Hence the transformations $g^{-1}ahu$ and $b$ coincide as claimed. This 
gives that $a\in \genset{G,b}\setminus G$ and $b\in\genset{G,a}\setminus G$. 
\end{proof}

Let $G$ be a  subgroup of $\sym$, let $1\leq i\leq n-1$, let $I_1, I_2, \ldots, I_m$ be representatives of orbits of $G$ on subsets of $\{1,2, \ldots, n\}$ with size $i$, let $K_1, K_2, \ldots, K_r$ be representatives of orbits of $G$ on the partitions of $\{1,2,\ldots, n\}$ with $i$ classes, let $f_{j,k}$ be an fixed arbitrary element of $\trans\setminus \sym$ with image $I_j$ and kernel $K_k$, let $T_j$ be a transversal of cosets of the stabilizer 
$G_{I_i}$ in $G$, and finally, let $F=\set{f_{j,k}t}{t\in T_{j}}$.
Then to verify that $\genset{G, a}\setminus G$ is idempotent generated for all $a\in\trans\setminus \sym$ with $\rank(a)=i$, it suffices, 
by Lemma \ref{reduce} and the preceding comments, to verify that $\genset{G, a}\setminus G$ is idempotent generated for all $a\in F$.

\section{Final remarks and problems}

We finish the paper by stating some  related open problems.

(i)
Is it possible to prove the main theorems of this paper without
using results that rely on the classification of finite simple groups?
In order to give a positive answer to this question, one has to give another 
proof for Theorem~\ref{superlemma} that does not use Mar\'oti's bound for
the order of a proper primitive group.

(ii) Classify the subgroups $G$ of $\sym$ that together with any singular 
transformation $a$ satisfy $\genset{G,a} =\genset{a^g \mid 
g\in G}$.
 
(iii)  Prove classification theorems analogous to Theorems~\ref{main} 
and~\ref{th2} 
for linear groups and to groups of automorphisms of independence 
algebras; see~\cite{Cameronszabo,gould} for the background theory of 
independence algebras.

(iv) 
Classify the pairs $(G,a)$, where $G\leq \sym$ and $a\in\trans$ such 
that $\genset{G,a}$ is regular.  
Consider also the corresponding problem for linear groups. 
As mentioned in the introduction, 
McAlister proved that for every idempotent $e\in \trans$ 
such that $\rank (e)=n-1$,
and for all $G\leq \sym$ the semigroup $\genset{G,e}$ is regular.

\medskip

{\bf Acknowledgement.} 
The research presented in this paper was partially supported by the 
FCT and FEDER project ISFL-1-143
of Centro de \'Algebra da Universidade de Lisboa, and by 
the FCT and PIDDAC 
projects PTDC/MAT/69514/2006 and PTDC/MAT/101993/2008.
The third author would like to acknowledge the support of the Hungarian 
Scientific Research Fund (OTKA) grant~72845.

%% \section*{Addresses}
%% \begin{tabular}{ll}
%% \emph{J. Ara\'ujo} \\\
%% Universidade Aberta, \\
%% R. da Escola Polit\'ecnica, 147,  \\
%% 1269-001 Lisboa,  \\
%% Portugal \\
%% \\
%% \\
%% \\
%% Centro de \'Algebra,&  \emph{J. D. Mitchell}\\
%% Universidade de Lisboa, & Mathematical Institute,\\
%% Av.\ Prof.\ Gama Pinto, 2, & North Haugh, \\
%% 1699 Lisboa Codex, & St Andrews, \\
%% Portugal &Fife KY16 9SS,\\
%% \emph{email}: mjoao@ptmat.lmc.fc.ul.pt& Scotland\\
%% & \emph{email}:jdm3@st-andrews.ac.uk\\
%% \end{tabular}

%%%%%%%
%%%%%%%

\def\cprime{$'$}

\end{document}